\documentclass{amsart}
\usepackage{amsmath}
\usepackage{amsfonts}
\usepackage{amssymb}
\usepackage{amsthm}
\usepackage{graphicx}
\usepackage{mathrsfs}
\usepackage[english]{babel}
\usepackage{fullpage}

\numberwithin{equation}{section}
\theoremstyle{definition}

\newtheorem{claim}{Claim}[subsection]
\newtheorem{conclusion}{Conclusion}[subsection]

\newtheorem{definition}{Definition}[subsection]

\newtheorem{remark}{Remark}[subsection]

\theoremstyle{plain}
\newtheorem{corollary}{Corollary}[subsection]
\newtheorem{lemma}{Lemma}[subsection]
\newtheorem*{acknowledgement}{Acknowledgement}

\newtheorem{proposition}{Proposition}[subsection]
\newtheorem{theorem}{Theorem}[subsection]
\newcommand{\im}{\textrm{im}}

\begin{document}

\title{Two-colored noncommmutative Gerstenhaber formality \\and infinity Duflo isomorphism}
\author{Johan Alm}
\date{}
\maketitle

\begin{abstract}
Using new configuration spaces, we give an explicit construction that extends Kontsevich's Lie-infinity quasi-isomorphism from polyvector fields to Hochschild cochains to a quasi-isomorphism of A-infinity algebras equipped with actions by homotopy derivations of the Lie algebra of polyvector fields. One may term this formality a formality of two-colored noncommutative Gerstenhaber homotopy algebras. In our result the action of polyvector fields by homotopy derivations of the wedge product on polyvector fields is not the adjoint action by the Schouten bracket, but a homotopy nontrivial and, in a sense, unique deformation of that action.

As an application we give an explicit Duflo-type construction for Lie-infinity algebras that generalizes the Duflo-Kontsevich isomorphism between the Chevalley-Eilenberg cohomology of the symmetric algebra on a Lie algebra and the Chevalley-Eilenberg cohomology of the universal enveloping algebra of the Lie algebra.
\end{abstract}

\section*{\bf{Introduction}}

Kontsevich's Formality Map is best understood as a morphism of two-colored operads
  \[
   \mathcal{K}(\overline{C}(\mathbf{H}))\rightarrow\mathcal{E}nd(T_{poly},\mathcal{O}),
  \]
where \(\mathcal{K}(\overline{C}(\mathbf{H}))\) is the operad of fundamental chains of a certain cellular operad \(\overline{C}(\mathbf{H})\) of compactified configuration spaces of points in the closed upper half-plane and \(\mathcal{E}nd(T_{poly},\mathcal{O})\) is the standard two-colored endomorphism operad on formal polyvector fields, \(T_{poly}\), and formal smooth functions, \(\mathcal{O}\), on some chosen graded vector space. The content of this map of operads is an \(L_{\infty}\) map from \(T_{poly}\) to the (differential) Hochschild cochain complex of \(\mathcal{O}\). In this note we introduce a three-colored operad \(\overline{CF}(\mathbf{H})\) of compactified configuration spaces of points in a closed upper half-plane equipped with a line parallel to the boundary, and, using the same techniques as Kontsevich, a representation 
  \[
   \mathcal{K}(\overline{CF}(\mathbf{H}))\rightarrow\mathcal{E}nd(T_{poly},T_{poly},\mathcal{O})
  \]
of its fundamental chains. This representation implies
\begin{itemize}
 \item Kontsevich's \(L_{\infty}\) map \(T_{poly}\rightarrow C(\mathcal{O},\mathcal{O})\) to the Hochschild cochain complex of the associative algebra of functions,
 \item an \(L_{\infty}\) map \(T_{poly}\rightarrow C^{\geq 1}(T_{poly},T_{poly})\) to the Hochschild cochain complex of the associative algebra of polyvector fields, extending the canonical adjoint action of \(T_{poly}\) on itself,
 \item and a morphism \(T_{poly}\rightarrow C(\mathcal{O},\mathcal{O})\) of \(A_{\infty}\) algebras equipped with actions of the Lie algebra \(T_{poly}\) by homotopy derivations.
\end{itemize}
This data can be concisely encoded as a quasi-isomorphism of two-colored noncommutative \(G_{\infty}\) algebras.

The three-colored operad \(\overline{CF}(\mathbf{H})\) is closely related to the moduli spaces of quilted holomorphic disks introduced in the context of Floer homology by Mau and Woodward in \cite{MW}. The moduli spaces of quilted holomorphic disks form a two-colored operad that can be embedded as a suboperad of our three-colored operad.

As an application we give an explicit strong homotopy version of the Duflo isomorphism. This generalizes earlier work by Calaque, Kontsevich, Manchon, Pevzner, Rossi, Torossian and others; see \cite{MT,PT,CR,Sh,Kon03}. More specifically, we construct a universal and generically homotopy nontrivial \(A_{\infty}\) deformation \(C(\mathbf{g},S(\mathbf{g}))_{exotic}\) of the Chevalley-Eilenberg cochain algebra \(C(\mathbf{g},S(\mathbf{g}))\) and an \(A_{\infty}\) quasi-isomorphism \(C(\mathbf{g},S(\mathbf{g}))_{exotic}\rightarrow C(\mathbf{g},U(\mathbf{g}))\) that on the cohomology level reproduces the Duflo-Kontsevich isomorphism of Chevalley-Eilenberg cohomologies. This implies that the Duflo-Kontsevich isomorphism can \textit{not} be universally lifted to an \(A_{\infty}\) quasi-isomorphism \(C(\mathbf{g},S(\mathbf{g}))\rightarrow C(\mathbf{g},U(\mathbf{g}))\) of the Chevalley-Eilenberg cochain algebras.
\begin{acknowledgement}
It is a pleasure to thank Johan Gran{\aa}ker, Carlo A.~Rossi, Bruno Vallette and Thomas Willwacher for discussions and helpful criticism. Special thanks to my supervisor, Sergei Merkulov.
\end{acknowledgement}

\subsection*{Conventions and notation}
Given a natural number \(n\), \([n]\) denotes the set \(\{1,2,\dots,n\}\). The cardinality of a finite set \(A\) is written \(\vert A\vert\), e.g. \(\vert [n]\vert=n\). Given finite sets \(A\) and \(B\), we write \(A+B\) for their disjoint union. We customarily write \(0\) for the empty set. If \(S\) is a subset of a finite set \(A\), we customarily write \(A-S\) for the complement of \(S\) in \(A\). We write \(A/S\) for the set \(A-S+\{S\}\). (So the cardinality of \(A/S\) is \(\vert A\vert -\vert S\vert +1\).) If \(A\) is an ordered finite set, then we say \(S\subset A\) is a connected subset, and write \(S<A\), if \(s,s''\in S\) and \(s<s'<s''\in A\) implies also \(s'\in S\). The group of permutations of a finite set \(T\) is denoted \(\Sigma_T\), and \(\Sigma_{[n]}\) is denoted \(\Sigma_n\).

The terms ``chain complex'' and ``dg vector space'' are used as synonyms and refer to (possibly unbounded) cohomologically graded chain complexes over the real numbers. The \(n\)-fold suspension of a chain complex \(V\) is the chain complex \(V[n]\) with \(V[n]^d=V^{d+n}\). We employ the Koszul symmetry conventions for the category of chain complexes. Given a dg vector space \(V\), \(T(V):=\prod_{n\geq 0}V^{\otimes n}\) and \(S(V):=\prod_{n\geq 0}(V^{\otimes n})_{\Sigma_n}\). The vector space \(S(V)\) has a structure of commutative associative algebra; we denote \(S(V)\) with this algebra structure \(S^a(V)\). It can also be regarded as a cocommutative coassociative coalgebra, which we denote \(S^c(V)\).

\section{\bf{Configuration space models for {$NCG_{\infty}$} algebras and flag OCHAs}}

In this section we define four different operads in the category of cellular compact smooth manifolds with corners. Two of the operads are our invention.

\subsection{A configuration space model for {$L_{\infty}$}.} 
For an integer \(\ell\geq 2\), let \(\mathrm{Conf}_{\ell}(\mathbf{C})\) be the manifold of all injective maps of \([\ell]:=\{1,\dots,\ell\}\) into \(\mathbf{C}\). The group of translations and positive dilations of the plane, \(\mathbf{C}\rtimes\mathbf{R}_{>0}\), acts on the plane and hence (by postcomposition) on \(\mathrm{Conf}_{\ell}(\mathbf{C})\). Define \(C_{\ell}(\mathbf{C}):=\mathrm{Conf}_{\ell}(\mathbf{C})/\mathbf{C}\rtimes\mathbf{R}_{>0}\). Let \(\overline{\mathrm{Conf}}_{\ell}(\mathbf{C})\) be the real Fulton-MacPherson compactification (in the literature also called the Axelrod-Singer compactification) of \(\mathrm{Conf}_{\ell}(\mathbf{C})\), i.e.~the real oriented blow-up of \(\mathbf{C}^{\ell}\) along all diagonals. The action by translations and positive dilations is smooth; hence extends uniquely to a smooth action on \(\overline{\mathrm{Conf}}_{\ell}(\mathbf{C})\). Define \(\overline{C}_{\ell}(\mathbf{C})\) to be the quotient of \(\overline{\mathrm{Conf}}_{\ell}(\mathbf{C})\) by this action. It is a smooth compact manifold with corners with codimension one boundary
\[
 \bigsqcup_S C_{\ell-\vert S\vert+1}(\mathbf{C})\times C_S(\mathbf{C})
\]
given by products labelled by subsets \(S\subset[\ell]\) (of cardinality \(2\leq\vert S\vert<\ell\)). Moreover, the closure of \(C_{\ell-\vert S\vert+1}(\mathbf{C})\times C_S(\mathbf{C})\) in \(\overline{C}_{\ell}(\mathbf{C})\) is the product \(\overline{C}_{\ell-\vert S\vert+1}(\mathbf{C})\times \overline{C}_S(\mathbf{C})\). This means that the family of spaces \(\overline{C}(\mathbf{C})=\{\overline{C}_{\ell}(\mathbf{C})\}\) together with the inclusions of boundary components and permutation actions by permutation of points assemble into the structure of an operad. We promote it to an operad of oriented manifolds as follows. Let \(C^{std}_{\ell}(\mathbf{C})\) be the submanifold of \(\mathrm{Conf}_{\ell}(\mathbf{C})\) consisting of configurations \(x\) satisfying \(\sum_{i=1}^{\ell}x_i=0\) and \(\sum_{i=1}^{\ell}\vert x_i\vert^2=1\). The manifolds \(C_{\ell}(\mathbf{C})\) and \(C^{std}_{\ell}(\mathbf{C})\) are diffeomorphic. The manifold \(\mathrm{Conf}_{\ell}(\mathbf{C})\) is canonically oriented; hence so is \(C^{std}_{\ell}(\mathbf{C})\). We orient \(C_{\ell}(\mathbf{C})\) by pulling back the orientation on \(C^{std}_{\ell}(\mathbf{C})\). Requiring Stokes' formula (without a sign) to hold defines an orientation of the compactification \(\overline{C}_{\ell}(\mathbf{C})\). It is easy to see that all permutations of \([\ell]\) preserve the orienation.

The boundary description describes a canonical stratification and the face complexes of the stratification of each component form an operad \(\mathcal{K}(\overline{C}(\mathbf{C}))\) that this is freely generated as a graded operad by the set \(\{[C_{\ell}(\mathbf{C})]\mid \ell\geq 2\}\) of ``fundamental chains''. It is well-known that representations of \(\mathcal{K}(\overline{C}(\mathbf{C}))\) in a dg vector space \(V\) are in one-to-one correspondence with \(L_{\infty}\) structures on the suspension \(V[1]\) of \(V\); see e.g.~\cite{GJ}.

\subsection{A configuration space model for OCHA} 
Set \(\mathbf{H}:=\mathbf{R}\times\mathbf{R}_{\geq 0}\). For integers \(m,n>0\), with \(2m+n\geq2\), let \(\mathrm{Conf}_{m,n}(\mathbf{H})\) be the manifold of injections of \([m]+[n]\) into \(\mathbf{H}\) that map \([n]\) into the boundary \(\mathbf{R}\times\{0\}\) of the half-plane and \([m]\) into the interior. The group of translations along the boundary and positive dilations, \(\mathbf{R}\times\mathbf{R}_{>0}\), acts (by postcomposition) on \(\mathrm{Conf}_{m,n}(\mathbf{H})\) and we let \(C_{m,n}(\mathbf{H})\) be the quotient of this action. The embedding
\[
 \mathrm{Conf}_{m,n}(\mathbf{H})\rightarrow\mathrm{Conf}_{2m+n}(\mathbf{C})
\]
defined by sending a configuration in \([m]+[n]\hookrightarrow\mathbf{H}\) to its orbit under complex conjugation induces an embedding
\[
 C_{m,n}(\mathbf{H})\rightarrow C_{2m+n}(\mathbf{C})\subset\overline{C}_{2m+n}(\mathbf{C}).
\]
The compactification \(\overline{C}_{m,n}(\mathbf{H})\) of \(C_{m,n}(\mathbf{H})\) was in\cite{Kon03} defined as the closure under this embedding. It is a smooth manifold with corners with \(n!\) connected components. Let \(\overline{C}^+_{m,n}(\mathbf{H})\) be the connected component that has the boundary points ``compatibly ordered'', by which we mean that if \(i<j\in [n]=\{1<\dots<n\}\), then the point labelled by \(i\) is before the point labelled by \(j\) on the boundary for the orientation of the boundary induced by the orientation of the half-plane. This gives us a permutation-equivariant identification \(\overline{C}_{m,n}(\mathbf{H})\cong \overline{C}^+_{m,n}(\mathbf{H})\times\Sigma_n\). The codimension one boundary of \(\overline{C}^+_{m,n}(\mathbf{H})\) is
 \[
  \bigsqcup_I \bigl(C^+_{m-\vert I\vert+1,n}(\mathbf{H})\times C_I(\mathbf{C})\bigr)
  \sqcup\bigsqcup_{S,T}\bigl(C^+_{m-\vert S\vert,n-\vert T\vert+1}(\mathbf{H})\times C^+_{S,T}(\mathbf{H})\bigr).
 \]
Here \(C^+_{m-\vert I\vert+1,n}(\mathbf{H})\) is the interior of \(\overline{C}^+_{m-\vert I\vert+1,n}(\mathbf{H})\), etc. The union is over all subsets \(I\subset[m]\) and subsets \(S\subset [m]\), \(T<[n]\) such that all involved spaces are defined. This description of the boundary extends, via the identification \(\overline{C}_{m,n}(\mathbf{H})\cong \overline{C}^+_{m,n}(\mathbf{H})\times\Sigma_n\), to boundary descriptions for all connected components, and defines the structure of a two-coloured operad on the collection \(\overline{C}(\mathbf{H}):=\{\overline{C}_{\ell}(\mathbf{C}),\overline{C}_{m,n}(\mathbf{H})\}\), the points in the interior being inputs of one color and the points on the boundary being inputs of another color. The spaces \(\overline{C}_{m,n}(\mathbf{H})\) are defined using embeddings into spaces fo the form \(\overline{C}_{\ell}(\mathbf{C})\), for which we have chosen orientations. We orient the spaces \(\overline{C}_{m,n}(\mathbf{C})\) by the pullback orientations of these embeddings.

The dg operad of face complexes of the stratification defined by the boundary decomposition is again generated by the fundamental chains. We denote this operad of fundamental chains \(\mathcal{K}(\overline{C}(\mathbf{H}))\). A representation of it is referred to as an \textbf{open-closed homotopy algebra}, see \cite{Ho, KS06}, henceforth abbreviated as an OCHA. An OCHA consists of a pair of dg vector spaces \(V\) and \(W\), an \(L_{\infty}\) structure on \(V[1]\), an \(A_{\infty}\) structure on \(W\), and an \(L_{\infty}\) morphism from \(V\) to the Hochschild cochain complex of \(W\).

We now define flag versions of the operads \(\overline{C}(\mathbf{C})\) and \(\overline{C}(\mathbf{H})\). 

\subsection{Flag version of $\overline{C}(\mathbf{C})$.} 
Since the affine group preserves collinearity and parallel lines it makes sense to say that some points in a configuration \(x\in C_{\ell}(\mathbf{C})\) are collinear on a line parallel to the real axis. For integers \(p\geq 0\) and \(q\geq 1\) with \(p+q\geq 2\), define \(CF_{p,q}(\mathbf{C})\subset C_{[p]+[q]}(\mathbf{C})\) to be the subset of configurations for which the points labelled by \([q]\) are collinear on a line parallel to the real axis. Define \(\overline{CF}_{p,q}(\mathbf{C})\) to be its closure inside \(\overline{C}_{p+q}(\mathbf{C})\). It has \(q!\) connected components. Let \(CF^+_{p,q}(\mathbf{C})\) denote the interior of the connected component that has the collinear points compatibly ordered, by which we mean that if \(i<j\in [q]=\{1<\dots<q\}\), then the point labelled by \(i\) is before the point labelled by \(j\) on their common line for the orientation of the line induced by the orientation of the plane. Then \(CF_{p,q}(\mathbf{C})\cong CF^+_{p,q}(\mathbf{C})\times\Sigma_q\). We deduce that the codimension one boundary of the corresponding compact connected component, \(\overline{CF}^+_{p,q}(\mathbf{C})\), is
\[
 \bigsqcup_I \bigl(CF^+_{p-\vert I\vert+1,q}(\mathbf{C})\times C_I(\mathbf{C})\bigr)
 \sqcup\bigsqcup_{S,T}\bigl(CF^+_{p-\vert S\vert,q-\vert T\vert+1}(\mathbf{C})\times CF^+_{S,T}(\mathbf{C})\bigr).
\]
The union is over all subsets \(I\subset [p]\), \(S\subset[p]\), \(T<[q]\) for which all involved spaces are defined. One can use the inclusions of boundary components to define a two-colored operad structure on the collection
 \[
  \overline{CF}(\mathbf{C}):=\{\overline{C}_{\ell}(\mathbf{C}),\overline{CF}_{p,q}(\mathbf{C})\},
 \]
in a way completely analogous the previously discussed operadic structure on \(\overline{C}(\mathbf{H})\).
\begin{definition}
We call \(\overline{CF}(\mathbf{C})\) the operad of configurations on flags in the plane.
\end{definition}
We orient the spaces of the form \(\overline{CF}_{p,q}(\mathbf{C})\) by the pullback orientations of the defining embeddings into \(\overline{C}_{p+q}(\mathbf{C})\). As before one then obtains a dg operad \(\mathcal{K}(\overline{CF}(\mathbf{C}))\) of fundamental chains. It is almost identical to the operad \(\mathcal{K}(\overline{C}(\mathbf{H}))\) of OCHAs: its representations also consist of an \(L_{\infty}\) algebra \(V[1]\), an \(A_{\infty}\) algebra \(W\) and an \(L_{\infty}\) morphism from \(V\) to the Hochschild cochain complex of \(W\). The difference lies in that the latter operad contains chains \([C_{m,n}(\mathbf{H})]\) with \(n=0\) while the former operad does not contain any chain of the form \([CF_{p,q}(\mathbf{C})]\) with \(q=0\). This means that the \(L_{\infty}\) map of an OCHA contains components \(V^{\otimes p}\rightarrow W\), so called curvature terms, whilst the \(L_{\infty}\) map of a \(\mathcal{K}(\overline{C}(\mathbf{H}))\)-representation can not, i.e.~it maps into the truncated Hochschild cochain complex \(C^{\geq 1}(W,W)\).
\begin{definition}
We call \(\mathcal{K}(\overline{CF}(\mathbf{C}))\) the operad of two-colored noncommutative \(G_{\infty}\) algebras.
\end{definition}
\begin{remark}
Define a \textbf{two-colored noncommutative Gerstenhaber algebra} to be a pair \((L,A)\), where \(L[1]\) is a dg Lie algebra and \(A\) is a dg assciative algebra, together with a dg Lie algebra morphism \(L[1]\rightarrow\mathrm{Der}(A)\). Such algebras are representations of an operad \(\mathcal{NCG}\) and \(\mathcal{K}(\overline{CF}(\mathbf{C}))\) is the cobar construction on the Koszul dual cooperad of \(\mathcal{NCG}\). We prove in an appendix that \(\mathcal{NCG}\) is Koszul. Thus \(\mathcal{K}(\overline{CF}(\mathbf{C}))\) indeed deserves to be called the operad of two-colored noncommutative \(G_{\infty}\) algebras.
\end{remark}
We shall abbreviate ``two-colored noncommutative \(G_{\infty}\) algebra'' as \(NCG_{\infty}\) algebra.
\subsection{Flag version of $\overline{C}(\mathbf{H})$.} 
There is also a flag version of the operad \(\overline{C}(\mathbf{H})\), defined as follows. Let \(k,m,n\geq 0\) be integers with \(2k+m+n\geq 1\) if \(m\geq 1\) and \(k+n\geq 2\) if \(m=0\). Let \(CF_{k,m,n}(\mathbf{H})\) be the subspace of \(C_{k+m,n}(\mathbf{H})\) consisting of all configurations wherein the points labelled by \([m]\) are collinear on a line parallel to the boundary. Denote by \(\overline{CF}_{k,m,n}(\mathbf{H})\) the closure inside \(\overline{C}_{k+m,n}(\mathbf{H})\). Let \(CF^+_{k,m,n}(\mathbf{H})\) denote connected component of \(CF_{k,m,n}(\mathbf{H})\) that has both the collinear points and the boundary points compatibly ordered, i.e.~if \(i<j\) in \([m]\), then \(x_i<x_j\) on their common line of collinearity, and if \(r<s\) in \([n]\), then \(x_r<x_s\) on the boundary. The codimension one boundary of its compactification, \(\overline{CF}^+_{k,m,n}(\mathbf{H})\), has the form
\[
 \bigsqcup_I \bigl(CF^+_{k-\vert I\vert+1,m,n}(\mathbf{H})\times C_I(\mathbf{C})\bigr)
 \sqcup\bigsqcup_{P,Q}\bigl(CF^+_{k-\vert P\vert,m-\vert Q\vert+1,n}(\mathbf{H})\times CF^+_{P,Q}(\mathbf{C})\bigr)\]
\[
 \sqcup\bigsqcup_{S,T,U}\bigl(CF^+_{k-\vert S\vert,m-\vert T\vert,n-\vert U\vert+1}(\mathbf{H})\times CF^+_{S,T,U}(\mathbf{H})\bigr).
\]
The union is over all subsets \(I,P,S\subset[k]\), \(Q,T<[m]\), \(S<[n]\) for which all involved spaces are defined. These boundary factorizations define an operad structure, but now in three colors, on the collection
 \[
\overline{CF}(\mathbf{H}):=\{\overline{C}_{\ell}(\mathbf{C}),\overline{CF}_{p,q}(\mathbf{C}),\overline{CF}_{k,m,n}(\mathbf{H})\}. \]
\begin{definition}
We call \(\overline{CF}(\mathbf{H})\) the operad of configurations on flags in the half-plane.
\end{definition}
Orient the spaces \(\overline{CF}_{k,m,n}(\mathbf{H})\) by the pullback orientations of the embeddings into \(\overline{C}_{k+m,n}(\mathbf{H})\). There is an associated operad \(\mathcal{K}(\overline{CF}(\mathbf{H}))\) of fundamental chains.
\begin{definition}
We call \(\mathcal{K}(\overline{CF}(\mathbf{H}))\) the operad of flag open-closed homotopy algebras, abbreviated as the operad of flag OCHAs.
\end{definition}
\begin{lemma}
A representation of the operad of flag open closed homotopy algebras in a triple \((L,A,B)\) of chain complexes is equivalent to
\begin{itemize}
 \item an \(NCG_{\infty}\) algebra structure on \((L,A)\);
 \item an OCHA structure on \((L,B)\);
 \item and a morphism from \(A\) to \(C(B,B)\) of \(A_{\infty}\) algebras with \(L_{\infty}\) actions of \(L\) by homotopy derivations, where the Hochschild cochain complex of \(B\) is considered with the \(L\)-action induced by the OCHA structure.
\end{itemize}
\end{lemma}
The first two listed items are obvious. Let \(\mathcal{M}or_*(\mathcal{NCG})_{\infty}\) be the Koszul resolution of the operad, \(\mathcal{M}or_*(\mathcal{NCG})\), whose representations are NCGAs \((L,A)\), \((L,A')\), with the same dg Lie algebra \(L\) appearing in both pairs, and a morphism between the two dg associative algebras respecting the actions by \(L\). See the appendix for some comments on why \(\mathcal{M}or_*(\mathcal{NCG})\) is Koszul. The third item in the list is a \(\mathcal{M}or_*(\mathcal{NCG})_{\infty}\)-representation on \((L,A,C(B,B))\). The key to this correspondence is to change from the operadic perspective that the chains \([CF_{k,m,n}(\mathbf{H})]\) are represented as maps \(L^{\otimes k}\otimes A^{\otimes m}\otimes B^{\otimes n}\rightarrow B\) to the perspective that they define maps
\[
 L^{\otimes k}\otimes A^{\otimes m}\rightarrow\mathrm{Map}(B^{\otimes n},B).
\]
(This hom-adjunction argument exactly parallels the argument used for interpreting an OCHA structure \(\{[C_{p,q}(\mathbf{H})]:L^{\otimes p}\otimes B^{\otimes q}\rightarrow B\}\) as an \(L_{\infty}\) morphism \(L\rightarrow C(B,B)\), compare with \cite{KS06,Ho}.) After this reinterpretation of the chains the argument reduces to (i) recognizing the induced \(NCG_{\infty}\) algebra structure on \((L,C(B,B))\) and (ii) comparing the differential on the chains to the differential on \(\mathcal{M}or_*(\mathcal{NCG})_{\infty}\). The details are left to the reader. We work out some more explicit details in the subsequent sections.
\begin{remark}
Consider the two-colored suboperad of \(\overline{CF}(\mathbf{H})\) on the components 
 \[
 \{\overline{CF}_{0,q}(\mathbf{C}),\overline{CF}_{0,m,0}(\mathbf{H}),\overline{CF}_{0,0,n}(\mathbf{H})\}.
 \]
It is isomorphic as an operad of smooth manifolds with corners to the operad of quilted holomorphic disks introduced by Mau and Woodward in \cite{MW}. Its operad of cellular chains is the operad of morphisms of \(A_{\infty}\) algebras.
\end{remark}
\section{\bf{(Co)operads of graphs}}

Kontsevich's proof of his Formality Conjecture and construction of a universal deformation quantization formula can be regarded\cite{Mer10} as the construction of
\begin{itemize}
 \item a map of cooperads \(\omega:\mathfrak{G}^c_{{\scriptscriptstyle \overline{C}(\mathbf{H})}}\rightarrow\Omega(\overline{C}(\mathbf{H}))\), where \(\mathfrak{G}^c_{{\scriptscriptstyle \overline{C}(\mathbf{H})}}\) is a cooperad of Feynman diagrams,
 \item and a map of operads \(\Phi:\mathfrak{G}_{{\scriptscriptstyle \overline{C}(\mathbf{H})}}\rightarrow\mathcal{E}nd(T_{poly},\mathcal{O})\) from the dual operad of Feynman diagrams.
\end{itemize}
Dualizing the map of cooperads and composing, one gets a representation
 \[
  \Phi\circ\omega^*:\mathcal{K}(\overline{C}(\mathbf{H}))\rightarrow \mathfrak{G}_{{\scriptscriptstyle \overline{C}(\mathbf{H})}}
\rightarrow\mathcal{E}nd(T_{poly},\mathcal{O})
 \]
of the fundamental chains of half-plane configurations, i.e.~an OCHA structure on \((T_{poly},\mathcal{O})\). We shall show that Kontsevich's construction can be extended, essentially without any changes, to a representation
\[
  \Phi\circ\omega^*:\mathcal{K}(\overline{CF}(\mathbf{H}))\rightarrow \mathfrak{G}_{{\scriptscriptstyle \overline{CF}(\mathbf{H})}}
\rightarrow\mathcal{E}nd(T_{poly},T_{poly},\mathcal{O})
 \]
of the operad of flag OCHAs. This is our \(NCG_{\infty}\) Formality Theorem. The new data added by extending Kontsevich's OCHA to a flag OCHA is a quasi-isomorphism \(T_{poly}\rightarrow C(\mathcal{O},\mathcal{O})\) of \(A_{\infty}\) algebras with homotopy actions by \(T_{poly}\). 

The first construction we need for our extension of the Kontsevich representation is a suitable operad \(\mathfrak{G}_{{\scriptscriptstyle \overline{CF}(\mathbf{H})}}\).

\subsection{Directed graphs.} 
Choose a finite set \(S\). Let \(fdgra^d_S\) be the set of all injectve functions \(\Gamma\) of the set \([d]\) into \((S\times S)-\Delta\), for \(\Delta\) the diagonal of \(S\). We refer to such a \(\Gamma\) as a directed graph with \(d\) edges on the set \(S\) and introduce the following terminology:
\begin{itemize}
\item \(E_{\Gamma}:=\im(\Gamma)\) is the set of edges of \(\Gamma\). We consider it as ordered by the given isomorphism with \([d]\). The element \(\Gamma(i)\in E_{\Gamma}\) is written \(e_i\) and referred to as the \(i\)th edge.
\item The function \(s_{\Gamma}:E_{\Gamma}\subset S\times S\rightarrow S\) given by projection onto the first factor \(S\) is called the source map of \(\Gamma\). The projection \(t_{\Gamma}:E_{\Gamma}\rightarrow S\) onto the second factor is called the target map of \(\Gamma\). An edge \(e\) is said to be directed from \(s_{\Gamma}(e)\) to \(t_{\Gamma}(e)\).
\item The set \(S\) is called the set of vertices of \(\Gamma\).
\item The valence of a vertex is the number of edges having that vertex as either source or target.
\item A connected component of \(\Gamma\) is a maximal (with respect to inclusions) subset \(E\subset E_{\Gamma}\) with the property that \(s_{\Gamma}(E)\cup t_{\Gamma}(E)\) and \(s_{\Gamma}(E_{\Gamma}-E)\cup t_{\Gamma}(E_{\Gamma}-E)\) are disjoint. A graph with a single connected component is said to be connected.
\end{itemize}
Let \(dgra^d_S\) be the subset of \(fdgra^d_S\) of connected graphs. There is a natural action of the permutation groups \(\Sigma_d\) and \(\Sigma_S\) on \(dgra^d_S\) by, respectively, reordering edges and permuting the vertices. Let \(sgn_d\) be the one-dimensional sign representation of \(\Sigma_d\). Define, for any finite set \(I\), of cardinality at least \(2\), the graded \(\Sigma_I\)-module
\[
\mathfrak{G}^c_{{\scriptscriptstyle \overline{C}(\mathbf{C})}}(I):=\bigoplus_{j\geq 0} (\mathbf{R}\langle dgra^d_I \rangle \otimes_{\Sigma_d} sgn_d)[-d].
\]

Elements of \(\mathfrak{G}^c_{{\scriptscriptstyle \overline{C}(\mathbf{C})}}(I)^d\) may be represented as (linear combinations of) connected graphs with \(d\) directed edges ordered up to an even permutation, \(\vert I\vert\) vertices labelled by \(I\), without double edges and without tadpoles (edges that begin and end at the same vertex).

For a finite set \(P\) and a nonempty finite set \(Q\), with \(\vert P\vert+\vert Q\vert\geq 2\), let \(dgra^d_{P,Q}\) be a copy of the subset of \(fdgra^d_{P+Q}\) consisting of those graphs which have no connected components \(E\subset E_{\Gamma}\) with \(s_{\Gamma}(E)\cup t_{\Gamma}(E)\subset P\), and put
\[
\mathfrak{G}^c_{{\scriptscriptstyle \overline{CF}(\mathbf{C})}}(P,Q):=\bigoplus_{d\geq 0}
 (\mathbf{R}\langle dgra^d_{P,Q} \rangle \otimes_{\Sigma_d} sgn_d)[-d].
\]
The vertices labelled by \(P\) of a graph in \(\mathfrak{G}^c_{{\scriptscriptstyle \overline{CF}(\mathbf{C})}}(P,Q)\) are called  \textbf{free vertices} and the vertices labelled by \(Q\) are called \textbf{collinear vertices}. Our restrictions informally say that there are no connected components with only free vertices.

Assume given a triple of finite sets \((K,M,N)\), with \(2\vert K\vert+\vert M\vert+\vert N\vert\geq 1\) if \(M\) is nonempty, and \(2\vert K\vert+\vert N\vert\geq 2\) if \(M\) is empty. Let \(dgra^d_{K,M,N}\) be a copy of the subset of \(dgra^d_{K,M+N}\) consisting of graphs \(\Gamma\) having no edge with source a vertex labelled by \(N\). Set
\[
 \mathfrak{G}^c_{{\scriptscriptstyle \overline{CF}(\mathbf{H})}}(K,M,N):=\bigoplus_{d\geq 0}
 (\mathbf{R}\langle dgra^d_{K,M,N} \rangle \otimes_{\Sigma_d} sgn_d)[-d].
\]
The vertices labelled by \(K\) of a graph in \(\mathfrak{G}^c_{{\scriptscriptstyle \overline{CF}(\mathbf{H})}}(K,M,N)\) are called \textbf{free vertices}, the vertices labelled by \(M\) are called \textbf{collinear vertices} and the vertices labelled by \(N\) are called \textbf{boundary vertices}. 

\subsection{(Co)operad structures} 
We shall now describe how the vector spaces of (equivalence classes of) graphs defined above assemble into cooperads.

Given \(\Gamma_2\in dgra^{d_2}_{S_2}\) and \(\Gamma\in dgra^{d}_{S}\), where \(d_2\leq d\) and \(S_2\subset S\), we define an \textbf{embedding of \(\Gamma_2\) as a full subgraph of \(\Gamma\)} to be an order-preserving inclusion \(f:[d_2]\hookrightarrow [d]\) which makes 
\[
[d_2]\hookrightarrow[d]\stackrel{\Gamma}{\rightarrow}S\times S  \;\;\mathrm{equal}\;\; 
[d_2]\stackrel{\Gamma_2}{\rightarrow}S_2\times S_2\subset S\times S.
\]
An embedding of \(\Gamma_2\) as a full subgraph of \(\Gamma\) is written \(f:\Gamma_2\hookrightarrow \Gamma\). Given an embedding \(f\) as above, we define \(\Gamma/\Gamma_2\in dgra^{d-d_2}_{S/S_2}\) to be the graph which, as a function, is the composition
\[
[d-d_2]\cong[d]-\im(f)\stackrel{\Gamma}{\rightarrow}S\times S \rightarrow (S/S_2)\times (S/S_2).
\]
Here the leftmost bijection is the unique order-preserving bijection and the rightmost arrow is given by the canonical projection of \(S\) onto \(S/S_2=S-S_2+\{S_2\}\) (sending elements of \(S_2\) to the element \(\{S_2\}\)). If \(\Gamma_1=\Gamma/\Gamma_2\), \(\Gamma_1\in dgra^{d_1}_{S_1+\{v\}}\) (so \(S_1=S-S_2\) and we identify the singleton sets \(\{v\}\) and \(\{S_2\}\)), then the embedding and the quotient define a bijection \([d_1]+[d_2]\rightarrow [d]\). This defines an order on \([d_1]+[d_2]\), using the order on \([d]\). This order on \([d_1]+[d_2]\) is related to the lexicographic order given by \([d_1]<[d_2]\) using a unique bijection. Define \(\epsilon(\Gamma_2,\Gamma,\Gamma_1)\) to be the sign of that bijection.
We may now define a cooperadic cocomposition
\[
 \mathfrak{G}^c_{{\scriptscriptstyle \overline{C}(\mathbf{C})}}(I_1+I_2)\rightarrow \mathfrak{G}^c_{{\scriptscriptstyle \overline{C}(\mathbf{C})}}(I_1+\{v\})\otimes \mathfrak{G}^c_{{\scriptscriptstyle \overline{C}(\mathbf{C})}}(I_2)
\]
by
\[
 \Gamma\mapsto\sum_{\Gamma_1=\Gamma/\Gamma_2}\epsilon(\Gamma_2,\Gamma,\Gamma_1)\Gamma_1\otimes\Gamma_2.
\]
The sum is over all embeddings of some \(\Gamma_2\) into \(\Gamma\).
\begin{conclusion}
The collection 
 \[
 \mathfrak{G}^c_{{\scriptscriptstyle \overline{C}(\mathbf{C})}}:=\{\mathfrak{G}^c_{{\scriptscriptstyle \overline{C}(\mathbf{C})}}(\ell)\}
 \] 
carries a cooperad structure. The componentwise linear dual, \(\mathfrak{G}_{{\scriptscriptstyle \overline{C}(\mathbf{C})}}:=\{\mathfrak{G}^c_{{\scriptscriptstyle \overline{C}(\mathbf{C})}}(\ell)^*\}\), is an operad.
\end{conclusion}
We define a full subgraph embedding of a graph \(\Gamma_2\in dgra^{d_2}_{P_2,Q_2}\) into a graph \(\Gamma\in dgra^d_{P,Q}\) exactly as before, except that we now require \(P_2\subset P\) and \(Q_2\subset Q\) (not just \(P_2+Q_2\subset P+Q\)). The quotient \(\Gamma/\Gamma_2\) is defined as before and regarded as an element of \(dgra^{d-d_2}_{P-P_2,Q/Q_2}\). The sign \(\epsilon(\Gamma_2,\Gamma,\Gamma_1)\) is also defined as before. With these conventions for subgraphs and quotients we can repeat above definition and get cocomposition maps
\[
 \mathfrak{G}^c_{{\scriptscriptstyle \overline{CF}(\mathbf{C})}}(P_1+P_2,Q_1+Q_2)\rightarrow
\mathfrak{G}^c_{{\scriptscriptstyle \overline{CF}(\mathbf{C})}}(P_1,Q_1+\{v\})\otimes
\mathfrak{G}^c_{{\scriptscriptstyle \overline{CF}(\mathbf{C})}}(P_2,Q_2).
\]
The definitions repeat word for word when \(\Gamma_2\in dgra^{d_2}_I\), \(\Gamma\in dgra^d_{P,Q}\) and \(I\subset P\), if we agree on the convention that now \(\Gamma/\Gamma_2\) belongs to \(dgra^{d-d_2}_{P/I,Q}\), defining cocompositions
\[
 \mathfrak{G}^c_{{\scriptscriptstyle \overline{CF}(\mathbf{C})}}(P+I,Q)\rightarrow\mathfrak{G}^c_{{\scriptscriptstyle \overline{CF}(\mathbf{C})}}(P+\{v\},Q)\otimes
\mathfrak{G}^c_{{\scriptscriptstyle \overline{C}(\mathbf{C})}}(I).
\]
\begin{conclusion}
The collection \(\mathfrak{G}^c_{{\scriptscriptstyle \overline{CF}(\mathbf{C})}}:=\{\mathfrak{G}^c_{{\scriptscriptstyle \overline{C}(\mathbf{C})}}(\ell),\mathfrak{G}^c_{{\scriptscriptstyle \overline{CF}(\mathbf{C})}}(p,q)\}\) carries a cooperad structure. The componentwise linear dual, \(\mathfrak{G}_{{\scriptscriptstyle \overline{CF}(\mathbf{C})}}:=\{\mathfrak{G}^c_{{\scriptscriptstyle \overline{C}(\mathbf{C})}}(\ell)^*,\mathfrak{G}^c_{{\scriptscriptstyle \overline{CF}(\mathbf{C})}}(p,q)^*\}\), is an operad.

With the evident conventions for how to color the new vertex obtained by collapsing an embedded subgraph the same formulas define a cooperad structure on the collection 
 \[
\mathfrak{G}^c_{{\scriptscriptstyle \overline{CF}(\mathbf{H})}}:=\{\mathfrak{G}^c_{{\scriptscriptstyle \overline{C}(\mathbf{C})}}(\ell),\mathfrak{G}^c_{{\scriptscriptstyle \overline{CF}(\mathbf{C})}}(p,q), \mathfrak{G}^c_{{\scriptscriptstyle \overline{CF}(\mathbf{H})}}(k,m,n)\}.
 \]
Its linear dual, denoted \(\mathfrak{G}_{{\scriptscriptstyle \overline{CF}(\mathbf{H})}}\), is an operad. 
\end{conclusion}

\subsection{de Rham field theory}
Given a pair of distinct indices \(i,j\in[k]+[m]+[n]\) we follow Kontsevich and define a function 
 \[
 \phi^h_{i,j}:CF_{k,m,n}(\mathbf{H})\rightarrow\mathbf{S}^1,\,
 x+\mathbf{R}\rtimes\mathbf{R}_{>0} \mapsto \mathrm{Arg}\biggl(\frac{x_j-x_i}{x_j-\overline{x}_i}\biggr).
 \]
Here a barred variable denotes the complex conjugate variable. The function is smooth and extends to a smooth function defined on the compactified configuration space. Let \(\vartheta\) be the homogeneous normalized volume form on \(\mathbf{S}^1\).

Given a graph \(\Gamma\in dgra^d_{k,m,n}\), define
 \[
 \omega_{\Gamma} := \wedge_{i=1}^d (\phi^h_{s_{\Gamma}(e_i),t_{\Gamma}(e_i)})^*\vartheta.
 \]
The form \(\omega_{\Gamma}\) is a smooth closed differential form of degree \(d\) on \(\overline{CF}_{k,m,n}(\mathbf{H})\). We extend \(\omega\) to a map of dg vector spaces \(\mathfrak{G}^c_{{\scriptscriptstyle \overline{CF}(\mathbf{H})}}(k,m,n)\rightarrow\Omega(\overline{CF}_{k,m,n}(\mathbf{H}))\).

Define similarly, for indices \(i,j\in[\ell]\), \(\phi_{i,j}:C_{\ell}(\mathbf{C})\rightarrow\mathbf{S}^1\) by
 \[
  \phi_{i,j}:x+\mathbf{C}\rtimes\mathbf{R}_{>0} \mapsto \mathrm{Arg}(x_j-x_i).
 \]
The function \(\phi\) extends to the compactification. For a graph \(\Gamma\in dgra^d_{\ell}\), let
 \[
  \omega_{\Gamma} := \wedge_{i=1}^d (\phi_{s_{\Gamma}(e_i),t_{\Gamma}(e_i)})^*\vartheta.
 \]
This allows us to define maps of dg vector spaces \(\omega:\mathfrak{G}^c_{{\scriptscriptstyle \overline{C}(\mathbf{C})}}(\ell)\rightarrow\Omega(\overline{C}_{\ell}(\mathbf{C}))\). By identifying \(\overline{CF}_{p,q}(\mathbf{C})\) with a subset of \(\overline{C}_{p+q}(\mathbf{C})\) and \(dgra^d_{p,q}\) with a subset of \(dgra^d_{p+q}\) we can use this to define maps of dg vector spaces \(\omega:\mathfrak{G}^c_{{\scriptscriptstyle \overline{CF}(\mathbf{C})}}(p,q)\rightarrow\Omega(\overline{CF}_{p,q}(\mathbf{C}))\) as well.

In all cases we interpret the form associated to a graph without edges as the function identically equal to \(1\).
\begin{claim}
The de Rham complex functor \(\Omega\) is only comonoidal up to quasi-isomorphism with respect to the usual tensor product of dg vector spaces. Hence \(\Omega(\overline{CF}(\mathbf{H}))\) is only a cooperad up to quasi-isomorphisms. This inconvenience can be ignored by working with a completed tensor product, regarding it, say, as a cooperad in the category of chain complexes of nuclear Fr\'echet spaces. Our mapping \(\omega:\mathfrak{G}^c_{{\scriptscriptstyle \overline{CF}(\mathbf{H})}}\rightarrow\Omega(\overline{CF}(\mathbf{H}))\) is a morphism of cooperads in this category of cooperads.
\end{claim}
We shall not prove this statement as it is a consequence of similar statements in \cite{Mer10}. 

\subsection{A representation of the operad of graphs}
Fix for the remainder of this section a graded vector space \(V\), assumed finite-dimensional in each degree.

Define the \textbf{formal smooth functions on} \(V\), denoted \(\mathcal{O}\), to be the completed symmetric algebra on \(V^*\). Define the \textbf{formal polyvector fields on} \(V\), to be  denoted \(T_{poly}\), as the completed symmetric algebra on \(V^*\oplus V[-1]\). Note that \(\mathcal{O}\) is a subalgebra of \(T_{poly}\).

Let \(\tau\) be the image of \(id_V\) under \(V\times V^*\rightarrow V\otimes V^*[1]\cong (V^*\otimes V[-1])^*\) and regard it as a map \(V^*\otimes V[-1]\rightarrow\mathbf{R}\). It extends uniquely to a derivation of \(T_{poly}\). This derivation defines an endomorphism (of degree \(-1\)) of \(T_{poly}\otimes T_{poly}\) which we again denote \(\tau\). The \textbf{Schouten bracket} on \(T_{poly}\) is the map
\[
 [\,,\,]_S:=m\circ\tau\circ(id+(21)),
\]
where \(m\) denotes the product on \(T_{poly}\). It is well-known that the Schouten bracket is a (degree \(-1\)) Lie bracket. Given a finite set \(S\) and distinct elements \(s,t\in S\), define \(\tau_{s,t}\) to be the endomorphism of \(T_{poly}^{\otimes S}\) acting as \(\tau\) on the ``\(s\)th times \(t\)th factors'' and as the identity on all others.

For a graph \(\Gamma\in dgra^d_{k,m,n}\), let
 \[
 \Phi_{\Gamma}:= \varepsilon\circ m\circ\bigcirc_{i=1}^d\tau_{s_{\gamma}(e_i),t_{\Gamma}(e_i)}:
T_{poly}^{\otimes k}\otimes T_{poly}^{\otimes m}\otimes \mathcal{O}^{\otimes n}\rightarrow \mathcal{O}.
 \]
Here \(\varepsilon\) is the projection of \(T_{poly}\) onto \(\mathcal{O}\) defined by the projection \(V^*\oplus V[-1]\rightarrow V^*\), we regard
 \[
  T_{poly}^{\otimes k}\otimes T_{poly}^{\otimes m}\otimes \mathcal{O}^{\otimes n}\subset T_{poly}^{\otimes k+m+n},
 \]
and \(m:T_{poly}^{k+m+n}\rightarrow T_{poly}\) is the product. For a graph \(\Gamma\in dgra^d_{\ell}\) we define
 \[
  \Phi_{\Gamma}:= m\circ\bigcirc_{i=1}^d\tau_{s_{\gamma}(e_i),t_{\Gamma}(e_i)}:T_{poly}^{\otimes\ell}\rightarrow T_{poly}.
 \]
For a graph \(\Gamma\in dgra^d_{p,q}\) we use the same formula,
 \[
  \Phi_{\Gamma}:= m\circ\bigcirc_{i=1}^d\tau_{s_{\gamma}(e_i),t_{\Gamma}(e_i)}:
T_{poly}^{\otimes p}\otimes T_{poly}^{\otimes q}\rightarrow T_{poly}.
 \]
\begin{claim}
One verifies that these definitions define a morphism of dg operads 
 \[
 \Phi:\mathfrak{G}_{{\scriptscriptstyle \overline{CF}(\mathbf{H})}}\rightarrow\mathcal{E}nd(T_{poly},T_{poly},\mathcal{O}).
 \] 
\end{claim}

\section{\bf{$NCG_{\infty}$ formality}}

Combining the previous subsections, we have a representation
 \[
  \Phi\circ\omega^*:\mathcal{K}(\overline{CF}(\mathbf{H}))\rightarrow \mathfrak{G}_{{\scriptscriptstyle \overline{CF}(\mathbf{H})}}
\rightarrow\mathcal{E}nd(T_{poly},T_{poly},\mathcal{O}). 
 \]
Since \(\mathcal{K}(\overline{CF}(\mathbf{H}))\) is quasi-free the representation consists of a family of maps, one for each generator of \(\mathcal{K}(\overline{CF}(\mathbf{H}))\), satisfying some quadratic identities coming from the boundary differential on \(\mathcal{K}(\overline{CF}(\mathbf{H}))\). We shall denote the components as follows:
\begin{itemize}
 \item \(\lambda_{\ell}:=\Phi\circ\omega^*([C_{\ell}(\mathbf{C})])\in\mathrm{Map}^{3-2\ell}(T_{poly}^{\otimes\ell},T_{poly})\), for \(\ell\geq 2\).
 \item \(\nu_p:=\Phi\circ\omega^*([CF^+_{0,q}(\mathbf{C})])\in \mathrm{Map}^{2-q}(T_{poly}^{\otimes q},T_{poly})\) for \(q\geq 2\).
 \item \(\mu_n:=\Phi\circ\omega^*([CF^+_{0,0,n}(\mathbf{H})])\in\mathrm{Map}^{2-n}(\mathcal{O}^{\otimes n},\mathcal{O})\) for \(n\geq 2\).
 \item \(\mathcal{V}_{p,q}:=\Phi\circ\omega^*([CF^+_{p,q}(\mathbf{C})])\in\mathrm{Map}^{2-2p-q}(T_{poly}^{\otimes p}\otimes T_{poly}^{\otimes q},T_{poly})\) for \(p,q\geq 1\).
 \item \(\mathcal{U}_{k,n}:=\Phi\circ\omega^*([CF^+_{k,0,n}(\mathbf{H})])\in\mathrm{Map}^{2-2k-n}(T_{poly}^{\otimes k}\otimes\mathcal{O}^{\otimes n},\mathcal{O})\) for \(k\geq 1\), \(n\geq 0\).
 \item \(\mathcal{Z}_{k,m,n}:=\Phi\circ\omega^*([CF^+_{k,m,n}(\mathbf{H})])\in\mathrm{Map}^{1-2k-m-n}(T_{poly}^{\otimes k}\otimes T_{poly}^{\otimes m}\otimes\mathcal{O}^{\otimes n},\mathcal{O})\) for \(k\geq 0\), \(m\geq 1\), \(n\geq 0\).
\end{itemize}
Recall that the Hochschild cochain complex \(C(A,A)\) of an \(A_{\infty}\) algebra \(A\) is the dg vector space
\[
 \prod_{r\geq 0}\mathrm{Map}(A[1]^{\otimes r},A).
\]
(This is the completed Hochschild cochain complex. As our results are of a formal nature we shall always work with completed complexes.) The brace operations on the Hochschild cochains complex are maps
 \[
  (\,)\{\dots\}_p : C(A,A)\otimes \bigotimes_{i=1}^p C(A,A)\rightarrow C(A,A),\; p\geq 1,
 \]
defined for \(x\in\mathrm{Map}(A[1]^{\otimes r},A)\), \(x_i\in\mathrm{Map}(A[1]^{\otimes r_i},A)\), \(1\leq i\leq p\leq r\), \(n=r+r_1+\dots+r_p-p\), by
 \[
  x\{x_1,\dots,x_p\}_p(a_1,\dots,a_n) 
= \sum_{1\leq i_1<\dots<i_p<r} \pm x(a_1,\dots,a_{i_1},x_1(a_{i_1},\dots),\dots,a_{i_p},x_p(a_{i_p},\dots),\dots,a_n).
 \]
The Gerstenhaber bracket on the Hochschild cochain complex is the operation
 \[
  [x,y]_G:=x\{y\}_1\pm y\{x\}_1.
 \]
It is a graded Lie bracket of degree \(-1\) in our grading on the Hochschild cochain complex. Denote by \(C^{\geq 1}(A,A)\) the subspace \(\prod_{r\geq 1}\mathrm{Map}(A[1]^{\otimes r},A)\). It is a graded Lie subalgebra. Set \((\,)\{\dots\}:=\sum_{p\geq 1}(\,)\{\dots\}_p\) and define
 \[
  br:C(A,A)\rightarrow C^{\geq 1}(C(A,A),C(A,A)),x\mapsto ()\{x\}_1+x\{\dots\}.
 \]
One verifies that this is a map of graded Lie algebras.

An \(A_{\infty}\) structure on \(A\) is a Maurer-Cartan element \(m=d+m_2+\dots\) in \(C^{\geq 1}(A,A)\). The differential \([m,\,]_G\) makes the Hochschild cochain complex a dg Lie algebra. It is also an \(A_{\infty}\) algebra with \(A_{\infty}\) structure the Maurer-Cartan element \(\cup^m:=br(m)\) of \(C^{\geq 1}(C(A,A),C(A,A))\). When \(A\) has a given \(A_{\infty}\) structure \(m\) we shall usually write \(C(m)\) for \(C(A,A)\) with differential \([m,\,]_G\).

The interpretation of the components of our representation of \(\mathcal{K}(\overline{CF}(\mathbf{H}))\) is that
\begin{itemize}
 \item \(\lambda=\{\lambda_{\ell}\}\) is an \(L_{\infty}\) structure on \(T_{poly}\).
 \item \(\nu=\{\nu_p\}\) is an \(A_{\infty}\) structure on \(T_{poly}\).
 \item \(\mu=\{\mu_n\}\) is an \(A_{\infty}\) structure on \(\mathcal{O}\).
 \item \(\mathcal{V}=\{\mathcal{V}_{p,q}\}\) is an \(L_{\infty}\) map \((T_{poly},\lambda)\rightarrow C^{\geq 1}(\nu)\).
 \item \(\mathcal{U}=\{\mathcal{U}_{k,n}\}\) is an \(L_{\infty}\) map \((T_{poly},\lambda)\rightarrow C(\mu)\).
 \item \(\mathcal{Z}=\{\mathcal{Z}_{k,m,n}\}\) is a morphism of \(A_{\infty}\) algebras
 \[
  (T_{poly},\nu,\mathcal{V})\rightarrow(C(\mu),\cup^{\mu},br\circ\mathcal{U})
 \]
equipped with homotopy actions by \((T_{poly},\lambda)\).
\end{itemize}
This description is a result of the interpretation of the operad of flag open-closed homotopy algebras. All the component maps have an explicit description as sums over graphs, e.g.
 \[
  \mathcal{V}_{p,q} = \sum_{[\Gamma]\in [dgra^{2p+q-2}_{p,q}]} \int_{\overline{CF}^+_{p,q}(\mathbf{C})}\omega_{\Gamma}\Phi_{\Gamma},
 \]
with \([dgra^{2p+q-2}_{p,q}]\) the set of equivalence classes of graphs under the \(\Sigma_{2p+q-2}\)-action by permutation of edges. We shall use this description to give a more detailed description of the component maps. The main tool is ``Kontsevich's vanishing lemma'':
\begin{lemma}\cite{Kon03}
Let \(X\) be a complex algebraic variety of dimension \(N\geq 1\), and \(Z_1, . . . ,Z_{2N}\) be rational functions on \(X\), not equal identically to zero. Let \(U\) be any Zariski open subset of \(X\) such that functions \(Z_{\alpha}\) are defined and non-vanishing on \(U\), and \(U\) consists of smooth points. Then the integral
 \[
 \int_{U(\mathbf{C})}\wedge_{\alpha=1}^{2N}d(\mathrm{Arg}(Z_{\alpha}))
 \]
is absolutely convergent, and equal to zero.
\end{lemma}
\subsection{The {$L_{\infty}$} structure {$\lambda$}}
We have
 \[
  \lambda_{\ell} = \sum_{[\Gamma]\in[dgra^{2\ell-3}_{\ell}]} \int_{\overline{C}_{\ell}(\mathbf{C})}\omega_{\Gamma}\Phi_{\Gamma}.
 \]
For \(\ell\geq 3\), \(C_{\ell}(\mathbf{C})\cong \mathbf{S}^1\times U\), with \(U=(\mathbf{C}\setminus\{0,1\})^{\ell-2}\setminus diagonals\). This identification can be obtained by using the translation freedom to fix the point labelled by \(1\), say, at the origin of \(\mathbf{C}\) and using the dilation freedom to put the point labelled by \(2\), say, on the unit circle \(\mathbf{S}^1\). Multiplying the remaining points by the inverse of the phase of the point labelled by \(2\) gives a point in \(U\). Using this description we can reduce every integral
 \[
  \int_{\overline{C}_{\ell}(\mathbf{C})}\omega_{\Gamma}
 \]
to an integral over a circle times an integral of the type appearing in Kontsevich's vanishing lemma. Hence all weights vanish for \(\ell\geq 3\). The configuration space \(\overline{C}_2(\mathbf{C})\) is a circle. The set of graphs \(dgra^1_2\) contains two elements; the graph with an edge from \(1\) to \(2\) and the graph with an edge from \(2\) to \(1\). Both graphs have weight \(1\). It follows that \(\lambda_2\) is the Schouten bracket. As all higher homotopies \(\lambda_{\geq 3}\) vanish, this means \(\lambda\) is the usual graded Schouten Lie algebra structure on \(T_{poly}\).

\subsection{The {$A_{\infty}$} structure {$\nu$}}
The \(A_{\infty}\) structure \(\nu\) has components
 \[
  \nu_p = \sum_{[\Gamma]\in [dgra^{p-2}_{0,p}]} \int_{\overline{CF}^+_{0,p}(\mathbf{C})}\omega_{\Gamma}\Phi_{\Gamma}.
 \]
The angle between collinear points is constant, so the differential form associated to a graph containing an edge connecting collinear vertices will be zero; hence no such graphs can contribute. It follows that the only graph which contributes is the graph with two vertices and no edge. The associated differential form is identically equal to one and we evaluate it on the one-point space \(\overline{CF}_{0,2}(\mathbf{C})\). It follows that \(\nu=\nu_2\) is the usual (wedge) product on \(T_{poly}\).

\subsection{The {$A_{\infty}$} structure {$\mu$}}
Arguing as in the preceeding paragraph one deduces that \(\mu=\mu_2\) is the usual product on \(\mathcal{O}\).

\subsection{The {$L_{\infty}$} map {$\mathcal{V}$}}
Since
 \[
  \mathcal{V}_{p,q} = \sum_{[\Gamma]\in [dgra^{2p+q-2}_{p,q}]} \int_{\overline{CF}^+_{p,q}(\mathbf{C})}\omega_{\Gamma}\Phi_{\Gamma}
 \]
and \(\overline{CF}^+_{p,1}(\mathbf{C})\cong\overline{C}_{p+1}(\mathbf{C})\), the argument regarding the \(L_{\infty}\) structure \(\lambda\) can be repeated to conclude that \(\mathcal{V}_{p,1}=0\) for \(p\geq 2\), while
 \[
  \mathcal{V}_{1,1}:T_{poly}\otimes T_{poly}\rightarrow T_{poly}, X\otimes\xi\mapsto [X,\xi]_S.
 \]
In other words, \(\mathcal{V}_{1,1}\) is the adjoint action \(T_{poly}\rightarrow\mathrm{Der}(T_{poly})\) of \(T_{poly}\) on itself by derivations of the wedge product.

Using the translation freedom to put the collinear point labelled by \(1\) at the origin and the collinear point labelled by \(2\) at \(1\) identifies \(CF^+_{p,2}(\mathbf{C})\) with \((\mathbf{C}\setminus\{0,1\})^p\setminus diagonals\), so that one may again use Kontsevish's vanishing lemma and conclude that \(\mathcal{V}_{p,2}=0\) for all \(p\geq 1\).

Reflection of the plane in the line of collinearity induces a diffeomorphism \(f\) of \(\overline{CF}^+_{p,q}(\mathbf{C})\). (Choosing representative configurations with the collinear points on the real axis identifies \(f\) with complex conjugation.) The map \(f\) preserves orientation if \((-1)^p\) is even and reverses it if it is odd. For \(\Gamma\in dgra^{2p+q-2}_{p,q}\), \(f^*\omega_{\Gamma}=(-1)^{2p+q-2}\omega_{\Gamma}=(-1)^q\omega_{\Gamma}\). Thus
 \[
  (-1)^p\int_{\overline{CF}^+_{p,q}(\mathbf{C})}\omega_{\Gamma} = (-1)^q\int_{\overline{CF}^+_{p,q}(\mathbf{C})}\omega_{\Gamma},
 \]
implying the integral is \(0\) whenever \(p\) and \(q\) have different parity, i.e.~whenever \(p+q\) is odd. This means that the first homotopy to \(\mathcal{V}_{1,1}\) is given by \(\mathcal{V}_{1,3}\). The angle between collinear points is constant, so the differential form associated to a graph containing an edge connecting collinear vertices will be zero. The set \(dgra^3_{1,3}\) contains a unique graph without edges connecting collinear vertices, up to direction and ordering of edges, namely the graph with a free vertex of valence three and three collinear vertices of valence one. Hence there are eight (equivalence classes of) graphs (corresponding to the \(2^3\) ways to direct the three edges) contributing to \(\mathcal{V}_{1,3}\). Each of these eight equivalence classes has a representative with the edges ordered so that \(e_i\) connects the free vertex with the collinear vertex labelled by \(i\), \(1\leq i\leq 3\). These representatives all have weight \(1/24\). It follows that
 \[\mathcal{V}_{1,3} =\frac{1}{24}
  m\circ\bigl(\tau_{1,4}\circ\tau_{1,3}\circ\tau_{1,2} + \tau_{1,4}\circ\tau_{1,3}\circ\tau_{2,1}
+ \tau_{1,4}\circ\tau_{3,1}\circ\tau_{1,2} + \tau_{4,1}\circ\tau_{1,3}\circ\tau_{1,2}
 \]
 \[
  + \tau_{4,1}\circ\tau_{3,1}\circ\tau_{1,2} + \tau_{4,1}\circ\tau_{1,3}\circ\tau_{2,1}
+ \tau_{1,4}\circ\tau_{3,1}\circ\tau_{2,1} + \tau_{4,1}\circ\tau_{3,1}\circ\tau_{2,1} \bigr)
 \]
as a map \(T_{poly}^{\otimes 1+3}\rightarrow T_{poly}\). (The first of the four copies of \(T_{poly}\) acts on the last three.)

\subsection{The {$L_{\infty}$} map {$\mathcal{U}$}}
The map \(\mathcal{U}\) is, by construction, Kontsevich's Formality Map. Recall that it's first Taylor component \(\mathcal{U}_1=\sum_{n\geq 0}\mathcal{U}_{1,n}\) is the Hochschild-Kostant-Rosenberg quasi-isomorphism.

\subsection{The map {$\mathcal{Z}$} of {$NCG_{\infty}$} algebras}
Since \(\overline{CF}^+_{0,1,n}(\mathbf{H})\) is isomorphic to \(\overline{CF}^+_{1,0,n}(\mathbf{H})\) and \(dgra^n_{0,1,n}\) is isomorphic to \(dgra^n_{1,0,n}\), for all \(n\), the maps \(\mathcal{Z}_{0,1,n}\) coincide with the maps \(\mathcal{U}_{1,n}\). Hence the first Taylor component of \(\mathcal{Z}\),
 \[
  \sum_{n\geq 0}\mathcal{Z}_{0,1,n}:T_{poly}\rightarrow C(\mu),
 \]
is the Hochschild-Kostant-Rosenberg (HKR) quasi-isomorphism. The higher components of Kontsevich's Formality Map \(\mathcal{U}\)  are homotopies measuring the failure of the HKR map to respect the Lie brackets. In the same way, the higher components of \(\mathcal{Z}\) are homotopies that keep track of the failure of the HKR map to respect the associative products and the respective actions of \(T_{poly}\) by homotopy derivations of said associative products. Since the first component is the HKR morphism, we get the following theorem:
\begin{theorem}[Main Theorem]\label{maintheorem}
The algebras \((T_{poly},\wedge,\mathcal{V})\) and \((C(\mathcal{O},\mathcal{O}),d_H+\cup,br\circ\mathcal{U})\) are quasi-isomorphic as \(A_{\infty}\) algebras with \(L_{\infty}\) actions by \((T_{poly},[\,,\,]_S)\). The map \(Z=\{Z_{k,m}=\sum_{n\geq 0}\mathcal{Z}_{k,m,n}\}_{k\geq 0,m\geq 1}\) is an explicit such quasi-isomorphism.
\end{theorem}
This statement implies the following \(A_{\infty}\) formality theorem:
\begin{corollary}
The algebras \((T_{poly},\wedge)\) and \((C(\mathcal{O},\mathcal{O}),d_H+\cup)\) are quasi-isomorphic \(A_{\infty}\) algebras. The map \(A=\{A_m:=\sum_{n\geq 0}\mathcal{Z}_{0,m,n}\}_{m\geq 1}\) is an explicit such quasi-isomorphism.
\end{corollary}
This result has already been demonstrated, but in a different way, by Shoikhet; see \cite{Sh}.

\subsection{Induced {$A_{\infty}$} maps}
An \(NCG_{\infty}\) algebra consists in an \(L_{\infty}\) algebra \((L,\lambda)\), an \(A_{\infty}\) algebra \((A,\nu)\) and an \(L_{\infty}\) morphism \(\mathcal{V}:L\rightarrow C^{\geq 1}(\nu)\). Let \(\hbar\) be a formal parameter. The map \(\mathcal{V}\) induces a map on the sets of Maurer-Cartan elements,
 \[
  \mathrm{MC}(L[[\hbar]])\rightarrow\mathrm{MC}(C(\nu)[[\hbar]]), \pi\mapsto \sum_{p\geq 1}\frac{1}{p!}\mathcal{V}_{p,q}((\hbar \pi)^{\otimes p},\,).
 \]
This gives us, for each Maurer-Cartan element \(\pi\) of \(L\), an \(A_{\infty}\) structure
 \[
  \nu^{\mathcal{V}(\pi)}_q:=\nu_q+\sum_{p\geq 1}\frac{1}{p!}\mathcal{V}_{p,q}((\hbar \pi)^{\otimes p},\,), \;q\geq 1,
 \]
on \(A[[\hbar]]\).

If \(\mathcal{Z}:(L,A,\lambda,\mathcal{V},\nu)\rightarrow (L,B,\lambda,\mathcal{U},\mu)\) is a morphism of \(NCG_{\infty}\) algebras (the same \(L_{\infty}\) algebra acting on both and we assume the \(NCG_{\infty}\) algebra morphism is the identity on the Lie-color), then, for any Maurer-Cartan element \(\pi\) of \(L[[\hbar]]\), we get an induced map of \(A_{\infty}\) algebras
 \[
  \mathcal{Z}^{\pi}: (A[[\hbar]],\nu^{\mathcal{V}(\pi)})\rightarrow(B[[\hbar]],\mu^{\mathcal{U}(\pi)})
 \]
by \(\mathcal{Z}^{\pi}_m:=\mathcal{Z}_{0,m}+\sum_{k\geq 0}\frac{1}{k!}\mathcal{Z}_{k,m}((\hbar\pi)^{\otimes k},\,)\). If \(\mathcal{Z}\) is a quasi-isomorphism, then \(\mathcal{Z}^{\pi}\) is as well.

Applying this general construction to our representation \(\Phi\circ\omega^*\) produces, for any Maurer-Cartan element \(\pi\in T_{poly}\) (i.e.~a Poisson bivector),
\begin{itemize}
 \item an \(A_{\infty}\) structure \(\nu^{\mathcal{V}(\pi)}\) on \(T_{poly}[[\hbar]]\) with \(\nu^{\mathcal{V}(\pi)}_1+\nu^{\mathcal{V}(\pi)}_2 = \hbar[\pi,\,]_S+\wedge\),
 \item the \(A_{\infty}\) cup product on the Hochschild cochains of \(\mathcal{O}[[\hbar]]\) corresponding to the Kontsevich star product \(\mu^{\mathcal{U}(\pi)}\) on \(\mathcal{O}[[\hbar]]\) defined by \(\pi\),
 \item and an \(A_{\infty}\) quasi-isomorphism \(\mathcal{Z}^{\pi}:(T_{poly}[[\hbar]],\nu^{\mathcal{V}(\pi)})\rightarrow C(\mu^{\mathcal{U}(\pi)})[[\hbar]]\).
\end{itemize}
We record this fact as a corollary.
\begin{corollary}
Let \(\pi\in T_{poly}\) be a Poisson structure. Then the \(A_{\infty}\) algebra \((T_{poly}[[\hbar]],\nu^{\mathcal{V}(\pi)})\) is quasi-isomorphic as an \(A_{\infty}\) algebra to the algebra of Hochschild cochains on \(\mathcal{O}[[\hbar]]\) equipped with the cup product corresponding to the Kontsevich star product defined by \(\pi\). The map \(\mathcal{Z}^{\pi}\) is an explicit such quasi-isomorphism.
\end{corollary}
\subsection{Homological properties of the exotic {$NCG_{\infty}$} algebra structure {$\mathcal{V}$}}
Let \(\mathcal{NCG}\) be the two-colored operad of noncommutative Gerstenhaber algebras and let \(f:\mathcal{NCG}\rightarrow\mathfrak{G}_{{\scriptscriptstyle \overline{CF}(\mathbf{C})}}\) be the map which sends the bracket to the (sum of) graph(s) \(e_{12}+e_{21}\in\mathfrak{G}_{{\scriptscriptstyle\overline{CF}(\mathbf{C})}}(2)=\mathfrak{G}_{{\scriptscriptstyle \overline{C}(\mathbf{C})}}(2)\), for \(e_{12}\) (\(e_{21}\)) the graph with vertices \(\{1,2\}\) and a single edge from \(1\) to \(2\) (from \(2\) to \(1\)), sends the product to the graph in \(\mathfrak{G}_{{\scriptscriptstyle \overline{CF}(\mathbf{C})}}(0,2)\) which has two vertices and no edge, and sends the action to the graph in \(\mathfrak{G}_{{\scriptscriptstyle \overline{CF}(\mathbf{C})}}(1,1)\) which is \(e_{12}+e_{21}\) with the vertices in different colors. The composition \(\Phi\circ f:\mathcal{NCG}\rightarrow\mathcal{E}nd(T_{poly},T_{poly})\) is the usual structure of NCGA on polyvector fields in terms of the wedge product and the Schouten bracket. The deformation complex of \(f\) is the mapping cone
\[
 \mathscr{C}:=\mathrm{Cone}(\mathrm{Def}(\mathcal{L}ie^1_{\infty}\rightarrow \mathfrak{G}_{{\scriptscriptstyle \overline{C}(\mathbf{C})}})[-1]\rightarrow
\mathrm{Def}(\mathcal{A}ss_{\infty}\rightarrow\textstyle{\int}\mathfrak{G}_{{\scriptscriptstyle \overline{CF}(\mathbf{C})}})).
\]
See the appendix for notation and further details. The complex \(\mathrm{Def}(\mathcal{L}ie^1_{\infty}\rightarrow \mathfrak{G}_{{\scriptscriptstyle \overline{C}(\mathbf{C})}})\) is a directed version of Kontsevich's graph complex, \(GC\), and quasi-isomorphic to it, as shown in \cite{W}. The operad \(\int\mathfrak{G}_{{\scriptscriptstyle \overline{CF}(\mathbf{C})}}\) is a directed version of the operad \(\mathcal{G}raphs\) used by Kontsevich in his proof in \cite{Kon99} of the formality of the little disks operad, and it is quasi-isomorphic to it\cite{W}. Thomas Willwacher has proved the following:
\begin{theorem}\cite{W}
\begin{itemize} 
 \item \(H^0(GC)\cong\mathfrak{grt}\) as a graded Lie algebra.
 \item \(H^1(\mathrm{Def}(\mathcal{A}ss_{\infty}\rightarrow\mathcal{G}raphs))\cong\mathfrak{grt}\oplus\mathbf{R}[-1]\) as a vector space, where \(\mathbf{R}[-1]\) is spanned by the class of the sum of graphs contributing to \(\mathcal{V}_{1,3}\).
 \item The map \(GC[-1]\rightarrow\mathrm{Def}(\mathcal{A}ss_{\infty}\rightarrow\mathcal{G}raphs)\) is injective on cohomology.
\end{itemize}
\end{theorem}
This theorem, together with the long exact sequence for our mapping cone, implies that \(H^{d+1}(\mathscr{C})\cong H^{d+1}(\mathrm{Def}(\mathcal{A}ss_{\infty}\rightarrow\mathcal{G}raphs))/H^d(GC)\). In particular, \(H^1(\mathscr{C})\) is one-dimensional, spanned by the sum of graphs entering \(\mathcal{V}_{1,3}\).

Using the representation \(\Phi\) we can push this statement to a universal (or, rather, generic) statement about structures on polyvector fields.
\begin{corollary}
The exotic \(NCG_{\infty}\) algebra structure \(\mathcal{V}\) on polyvector fields is generically not homotopic to the usual such structure. Moreover, it represents the unique infinitesimal deformation of the usual structure.
\end{corollary}
\begin{corollary}
The \(A_{\infty}\) structures \(\wedge+\hbar[\pi,\,]\) and \(\nu^{\mathcal{V}(\pi)}\) on \(T_{poly}[[\hbar]]\) are, generically, not homotopic.
\end{corollary}
We have to say generically because for some dimensions of the \(\mathcal{O}\)-module \(T_{poly}\) and for some degenerate Maurer-Cartan elements the corollaries might not be true.

\section{\bf{A Duflo-type theorem}}

Kontsevich's paper \cite{Kon03} contained a (somewhat sketchy) proof that the tangential morphism of his Formality map, applied to a finite dimensional Lie algebra, defined an isomorphism \(H(\mathbf{g},S(\mathbf{g}))\rightarrow H(\mathbf{g}, U(\mathbf{g}))\) of Chevalley-Eilenberg cohomology algebras. This result was later given a detailed proof and generalized to an arbitrary dg Lie algebra of finite type, see \cite{MT,PT,CR}. In this section we discuss a homotopy generalization of this theorem.

Let \(\mathbf{g}\) be a graded vector space of finite type (i.e.~finite dimensional in each degree). Let \(T_{poly}\) be the polyvector fields on \(\mathbf{g}[1]\), so \(T_{poly}=S(\mathbf{g}^*[-1])\otimes S(\mathbf{g})\). (We shall assume all tensor algebras to be completed.) Identify \(T_{poly}\) with \(\mathrm{Map}(S(\mathbf{g}[1]),S(\mathbf{g}))\). The graded Lie algebra
 \[
 \mathrm{Def}(\mathcal{L}ie,\mathbf{g})[-1]:=\mathrm{Def}(\mathcal{L}ie_{\infty}\stackrel{0}{\rightarrow}\mathcal{E}nd(\mathbf{g}))[-1]
=\mathrm{Map}(S^{\geq 1}(\mathbf{g}[1])),\mathbf{g})
 \]
embeds into \(T_{poly}\) as a Lie subalgebra. We have \(\mathcal{O}=S(\mathbf{g}^*[-1])\), and
 \[
  C(\mathcal{O},\mathcal{O})=\mathrm{Map}(\mathrm{B}(S^a(\mathbf{g}^*[-1])),S^a(\mathbf{g}^*[-1]))\cong
\mathrm{Map}(S^c(\mathbf{g}[1]),\mathrm{\Omega}(S^c(\mathbf{g}[1]))).
 \]
Here \(\mathrm{B}(\,)\) denotes the (coassociative) bar construction and \(\mathrm{\Omega}(\,)\) denotes the (associative) cobar construction.

After the above identifications the following result is a straight-forward corollary to our Main Theorem, \ref{maintheorem}.
\begin{theorem}
The representation \(\Phi\circ\omega^*:\mathcal{K}(\overline{CF}(\mathbf{H}))\rightarrow\mathcal{E}nd(T_{poly},T_{poly},\mathcal{O})\) induces an explicit quasi-isomorphism \(\mathrm{Map}(S^c(\mathbf{g}[1]),S^a(\mathbf{g}))\rightarrow\mathrm{Map}(S^c(\mathbf{g}[1]),\mathrm{\Omega}(S^c(\mathbf{g}[1])))\) of \(A_{\infty}\) algebras equipped with \(L_{\infty}\) actions by the graded Lie algebra \(\mathrm{Def}(\mathcal{L}ie,\mathbf{g})\).
\end{theorem}
As before we can, given a Maurer-Cartan element of \(\mathrm{Def}(\mathcal{L}ie,\mathbf{g})\) push this to a quasi-isomorphism of the induced A-infintity structures. The formal parameter \(\hbar\) may in the present case be discarded (set to \(1\)). It's purpose is only to define filtrations that ensure we never encounter diverging sums, but in the present case one may use weight grading by tensor lengths to define such filtrations. It is a standard argument and we omit the details.

A Maurer-Cartan element \(Q\) of \(\mathrm{Def}(\mathcal{L}ie,\mathbf{g})\) is precisely an \(L_{\infty}\) structure on \(\mathbf{g}\). Assume \(Q\) given and interpret it as a coderivation of \(S^c(\mathbf{g}[1])\), and denote the dg coalgebra \((S^c(\mathbf{g}[1]),Q)\) by \(C(\mathbf{g})\). The cobar construction
 \[
  \mathrm{\Omega}(C(\mathbf{g}))=:U_{\infty}(\mathbf{g})
 \]
is the derived universal enveloping algebra of the \(L_{\infty}\) algebra \((\mathbf{g},Q)\) introduced by V. Baranovsky in \cite{B07}. (In the special case of a dg Lie algebra it is quasi-isomorphic to the usual universal enveloping algebra.) Kontsevich's formality map \(\mathcal{U}\) quantizes \(Q\) to a differential on \(S^a(\mathbf{g}^*[-1])\). Denote \(S^a(\mathbf{g}^*[-1])\) equipped with this differential by \(C(\mathbf{g},\mathbf{R})\). We have an isomorphism of algebras
 \[
  C(C(\mathbf{g},\mathbf{R}),C(\mathbf{g},\mathbf{R}))\cong\mathrm{Map}(C(\mathbf{g}),U_{\infty}(\mathbf{g}))=:C(\mathbf{g},U_{\infty}(\mathbf{g})).
 \]
However, the induced \(A_{\infty}\) structure on \(\mathrm{Map}(S^c(\mathbf{g}[1]),S^a(\mathbf{g}))\) is not simply
 \[
  C(\mathbf{g},S(\mathbf{g}))=\mathrm{Map}(C(\mathbf{g}),S^a(\mathbf{g})).
 \]
Instead, we obtain an \(A_{\infty}\) algebra \(C(\mathbf{g},S(\mathbf{g}))_{exotic}\), which is a (generically) homotopy nontrivial deformation of \(C(\mathbf{g},S(\mathbf{g}))\). The induced \(A_{\infty}\) quasi-isomorphism is
 \[
  \mathcal{Z}^Q:C(\mathbf{g},S(\mathbf{g}))_{exotic}\rightarrow C(\mathbf{g},U_{\infty}(\mathbf{g})).
 \]
\begin{remark}
\begin{itemize}
 \item The cohomologies \(H(C(\mathbf{g},S(\mathbf{g}))_{exotic})\) and \(H(\mathbf{g},S(\mathbf{g}))\) are isomorphic as associative algebras and the map on cohomology induced by \(\mathcal{Z}^Q\) coincides, by construction, with the Duflo-Kontsevich isomorphism. Thus our theorem generalizes the Duflo-Kontsevich statement.
 \item Since \(C(\mathbf{g},S(\mathbf{g}))_{exotic}\) is, generically, not quasi-isomorphic to \(C(\mathbf{g},S(\mathbf{g}))\), but--by our theorem--is quasi-isomorphic to \(C(\mathbf{g},U_{\infty}(\mathbf{g}))\), it follows that there does not, generically, exist a quasi-isomorphism of \(A_{\infty}\) algebras
 \[
  C(\mathbf{g},S(\mathbf{g})) \rightarrow C(\mathbf{g},U_{\infty}(\mathbf{g})).
 \]
In other words, it is impossible to find a universal \(A_{\infty}\) lift of the Duflo-Kontsevich isomorphism on Chevalley-Eilenberg cohomologies to the Chevalley-Eilenberg cochain algebras.
\end{itemize}
\end{remark}
There is a canonical isomorphism between \(T_{poly}\) on \(\mathbf{g}[1]\) and \(T_{poly}\) on \(\mathbf{g}^*\). Above we used the first graded vector space, for which \(\mathcal{O}=S^a(\mathbf{g}^*[-1])\). Application of Kontsevich's formality to the second case, for which \(\mathcal{O}=S^a(\mathbf{g})\), quantizes an \(L_{\infty}\) structure \(Q\in T_{poly}\) to a (flat) \(A_{\infty}\) structure \(\star\) on \(S(\mathbf{g})[[\hbar]]\). Calaque, Felder, Ferrario and Rossi constructed in \cite{CFFR} a nontrivial but explicit \(A_{\infty}\) \((S(\mathbf{g})[[\hbar]],\star)-C(\mathbf{g},\mathbf{R})[[\hbar]]\)-bimodule structure \(K_{\hbar}\) on \(\mathbf{R}[[\hbar]]\) and they proved that the derived left action
 \[
  L:(S(\mathbf{g})[[\hbar]],\star)\rightarrow\mathrm{Map}_{\hbar}(K_{\hbar}[1]\otimes \mathrm{B}(C(\mathbf{g},\mathbf{R}))[[\hbar]],K_{\hbar}[1])
 \]
is a quasi-isomorphism of \(A_{\infty}\) algebras. One may formally set \(\hbar=1\) in this quasi-isomorphism, for essentially the same reasons as those which allowed us to do so above then identify the term on the right (above), with the cobar construction \(\mathrm{\Omega}(C(\mathbf{g}))\). Thus the result of \cite{CFFR} implies that the quantization of the symmetric algebra on the \(L_{\infty}\) algebra \(\mathbf{g}\), \((S(\mathbf{g}),\star)\), is quasi-isomorphic to Baranovsky's derived universal enveloping algebra of \(\mathbf{g}\). A detailed proof of this will be contained in \cite{AR}. Together with our result this quasi-isomorphism implies that the \(A_{\infty}\) algebras \(C(\mathbf{g},S(\mathbf{g}))_{exotic}\) and \(C(\mathbf{g},(S(\mathbf{g}),\star))\) are quasi-isomorphic, though the quasi-isomorphism is presently not explicit.

\appendix

\section{{\bf{$NCG_{\infty}$} algebras}}

Let \(\mathcal{NCG}\) be the two-colored operad generated by a degree \(-1\) Lie bracket \([x_1,x_2]\) in one color, call it \(\mathbf{x}\), an associative degree \(0\) product \(a_1\cdot a_2\) in another color, call it \(\mathbf{a}\), and an operation which we denote \(x_1\bullet a_1\), of the type \((\mathbf{x},\mathbf{a})\rightarrow\mathbf{a}\), which represents the bracket in derivations of the product. This is the operad of NCGAs.
\begin{proposition}
The operad \(\mathcal{NCG}\) is Koszul.
\end{proposition}
\begin{proof}
We shall use the rewriting systems method of \cite{LV}. The rewriting rules are
\begin{align}
(a_1\cdot a_2)\cdot a_3 &\mapsto a_1\cdot (a_2\cdot a_3) \nonumber \\
[[x_1,x_2],x_3] &\mapsto -[[x_2,x_3],x_1]-[[x_3,x_1],x_2] \nonumber \\
x_1\bullet(a_1\cdot a_2) &\mapsto (x_1\bullet a_1)\cdot a_2 + a_1\cdot(x_1\bullet a_2) \nonumber\\
[x_1,x_2] \bullet a_1 &\mapsto x_1\bullet(x_2\bullet a_1) - x_2\bullet(x_1\bullet a_1). \nonumber
\end{align}
The critical monomials are \(((a_1\cdot a_2)\cdot a_3)\cdot a_4\), \([[[x_1,x_2],x_3],x_4]\), \(x_1\bullet ((a_1\cdot a_2)\cdot a_3)\), \([x_1,x_2] \bullet(a_1\cdot a_2)\) and \([[x_1,x_2],x_3] \bullet a_1\). The first two are known to be confluent as the operads \(\mathcal{L}ie^1\) and \(\mathcal{A}ss\) are known to be Koszul. The third critical monomial can be rewritten either as
\begin{align}
x_1\bullet ((a_1a_2)a_3) &\mapsto (x_1\bullet(a_1a_2))a_3+(a_1a_2)(x_1\bullet a_3) \nonumber \\
&\mapsto ((x_1\bullet a_1)a_2)a_3+(a_1(x_1\bullet a_2))a_3+a_1(a_2(x_1\bullet a_3)) \nonumber \\
&\mapsto (x_1\bullet a_1)(a_2a_3)+a_1((x_1\bullet a_2)a_3)+a_1(a_2(x_1\bullet a_3)) \nonumber
\end{align}
or as
\begin{align}
x_1\bullet ((a_1a_2)a_3) &\mapsto x_1\bullet(a_1(a_2a_3)) \nonumber \\
&\mapsto (x_1\bullet a)(a_2a_3)+a_1(x_1\bullet(a_2a_3)) \nonumber \\
&\mapsto (x_1\bullet a_1)(a_2a_3)+a_1((x_1\bullet a_2)a_3)+a_1(a_2(x_1\bullet a_3)). \nonumber
\end{align}
Since both ways give the same end result, \(x_1\bullet ((a_1a_2)a_3)\) is confluent.

The critical monomial \([x_1,x_2] \bullet(a_1\cdot a_2)\) can be rewritten either as
\begin{align}
[x_1,x_2] \bullet(a_1a_2) &\mapsto ([x_1,x_2]\bullet a_1)a_2+a_1([x_1,x_2]\bullet a_2) \nonumber \\
&\mapsto (x_1\bullet(x_2\bullet a_1))a_2-(x_2\bullet(x_1\bullet a_1))a_2+a_1(x_1\bullet(x_2\bullet a_2))-a_1(x_2\bullet(x_1\bullet a_2)) \nonumber
\end{align}
or
\begin{align}
[x_1,x_2] \bullet(a_1a_2) &\mapsto x_1\bullet(x_2\bullet(a_1a_2))-x_2\bullet(x_1\bullet(a_1a_2)) \nonumber \\
&\mapsto x_1\bullet((x_2\bullet a_1)a_2)+x_1\bullet(a_1(x_2\bullet a_2))-x_2\bullet((x_1\bullet a_1)a_2)-x_2\bullet(a_1(x_1\bullet a_2)) \nonumber \\
&\mapsto (x_1\bullet(x_2\bullet a_1))a_2+(x_2\bullet a_1)(x_1\bullet a_2)+(x_1\bullet a_1)(x_2\bullet a_2) + a_1(x_1\bullet(x_2\bullet a_2)) \nonumber \\
&-(x_2\bullet(x_1\bullet a_1))a_2-(x_1\bullet a_1)(x_2\bullet a_2)-(x_2\bullet a_1)(x_1\bullet a_2)-a_1(x_2\bullet(x_1\bullet a_2)). \nonumber
\end{align}
These two ways to rewrite the monomial agree, so it is also confluent. Confluence of the last critical monomial, \([[x_1,x_2],x_3] \bullet a_1\), is a similar straightforward manipulation and we omit it.
\end{proof}

For a Koszul operad \(\mathcal{P}\) the operad \(\mathcal{M}or(\mathcal{P})\), whose representations are pairs of \(\mathcal{P}\)-algebras together with a morphism of \(\mathcal{P}\)-algebras between them, is again Koszul by the results of \cite{MV09}. An algebra for the resolution \(\mathrm{\Omega}(\mathcal{M}or(\mathcal{P})^{\text{!`}})\) consists in two strong homotopy \(\mathcal{P}\)-algebras and a strong homotopy morphism between them. This general machinery produces a four-colored operad \(\mathrm{\Omega}(\mathcal{M}or(\mathcal{NCG})^{\text{!`}})\) of morphisms of \(NCG_{\infty}\) algebras. It has two ``Lie-colors'' and two ``Ass-colors''. We can make it into a 3-colored operad by identifying the two Lie colors. (A representation of that new operad will be a morphism of \(NCG_{\infty}\) algebras having the same \(L_{\infty}\) algebra acting on both \(A_{\infty}\) algebras.) This operad includes generators describing an \(L_{\infty}\) endomorphism of the Lie-color. Quotient out these generators and get a new three-colored operad \(\mathcal{M}or_*(\mathcal{NCG})_{\infty}\). Its representations are morphisms of \(NCG_{\infty}\) algebras that have the same \(L_{\infty}\) algebra acting on both \(A_{\infty}\) algebras, and for which the \(L_{\infty}\) endomorphism is the identity. It is easy to see, knowing that \(\mathrm{\Omega}(\mathcal{M}or(\mathcal{NCG})^{\text{!`}})\rightarrow\mathcal{M}or(\mathcal{NCG})\) is a quasi-isomorphism, that \(\mathcal{M}or_*(\mathcal{NCG})_{\infty}\) is quasi-isomorphic to the operad \(\mathcal{M}or_*(\mathcal{NCG})\) which has as representations two NCGAs with the same Lie algebra acting on both associative algebras and a morphism between the NCGAs which is the identity on the Lie algebra. Finally one can note that, in fact, \(\mathcal{M}or_*(\mathcal{NCG})_{\infty}=\mathrm{\Omega}(\mathcal{M}or_*(\mathcal{NCG})^{\text{!`}})\).
\begin{remark}
Consider the operad \(\mathcal{NCG}^{(1)}\) of (one-colored) noncommutative Gerstenhaber algebras (chain complexes that are simultaneously a dg Lie algebra, with the bracket of degree \(-1\), and an associative algebra, and the  Lie bracket acts by derivations of the associative product). Our method to prove Koszulity of \(\mathcal{NCG}\) does not repeat mutatis mutandum for \(\mathcal{NCG}^{(1)}\). The problem is that one gets a new critical monomial, \([x_1x_2,x_3x_4]\), which is not confluent. This suggests (but does not prove) that \(\mathcal{NCG}^{(1)}\) is not Koszul.
\end{remark}
\subsection{The deformation complex of {$\mathcal{NCG}_{\infty}\rightarrow\mathcal{P}$}}
Recall that we denote the two colors of \(\mathcal{NCG}\) by \(\mathbf{x}\) and \(\mathbf{a}\). Here \(\mathbf{x}\) is the ``Lie'' color and \(\mathbf{a}\) is the ``Ass'' color. Define \(\mathcal{NCG}_{\infty}:=\mathrm{\Omega}(\mathcal{NCG}^{\text{!`}})\).

Let \(\mathcal{P}\) be a dg operad with colors \(\mathbf{x}\) and \(\mathbf{a}\) and assume given a morphism of operads \(f:\mathcal{NCG}_{\infty}\rightarrow\mathcal{P}\). We shall describe the deformation complex \(\mathrm{Def}(f)\). 

We shall simplify notation and write \(\mathcal{P}(k)\) for \(\mathcal{P}(k,0;\mathbf{x})\) and \(\mathcal{P}(p,q)\) for \(\mathcal{P}(p,q;\mathbf{a})\). As a chain complex,
 \[
  \mathrm{Def}(f) =\prod_{k\geq 2} \mathcal{P}(k)_{\Sigma_k}[2-2k] \oplus 
\prod_{\substack{p\geq 0,q\geq 1\\p+q\geq 2}} \mathcal{P}(p,q)_{\Sigma_p}\otimes sgn_q[1-2p-q].
 \]
This chain complex has a degree zero graded Lie bracket defined by taking the commutator of operadic composition. The map \(f=(f_k)+(f_{p,q})\) is a Maurer-Cartan-element and the differential on the deformation complex is the internal differential on \(\mathcal{P}\) plus the bracket \([f,\,]\). We can give a more suggestive formulation of the deformation complex as follows. The components \((f_k)\) define a morphism \(\lambda^f:\mathcal{L}ie^1_{\infty}\rightarrow\mathcal{P}\) and
 \[
  \mathrm{Def}(\lambda^f)=\prod_{k\geq 2} \mathcal{P}(k)_{\Sigma_k}[2-2k]
 \]
with differential (the internal differential on \(\mathcal{P}\) plus) \([(f_k),\,]\). Set
 \[
  {\textstyle{\int}}\mathcal{P}(q):=\prod_p\mathcal{P}(p,q)_{\Sigma_p}[-2p].
 \]
The collection \(\int\mathcal{P}=\{\int\mathcal{P}(q)\}\) has a structure of dg operad. (This can actually be interpreted as a categorical end: a \(\Sigma\)-bimodule can be regarded as a bifunctor and we take the limit over one argument.) The compositions of the Lie-color in \(\mathcal{P}\) define a right action \(\bullet\) of \(\mathrm{Def}(\lambda^f)\) on \(\int\mathcal{P}\) by operadic derivations. Add to the differential the term \([(f_{p,1}),\,]+(\,)\bullet(f_k)\). The remaining mixed components of \(f\), i.e.~\((f_{p,q})\) with \(q\geq 2\), define a morphism \(\mu^f:\mathcal{A}ss_{\infty}\rightarrow\int\mathcal{P}\) with \(\mu^f_q=(f_{p,q})_{p\geq 0}\). We have
 \[
  Def(\mu^f)=\prod_{\substack{p\geq 0,q\geq 1\\p+q\geq 2}} \mathcal{P}(p,q)_{\Sigma_p}\otimes sgn_q[1-2p-q].
 \]
The components \((f_{p,q})_{p\geq 1}\) define a map of complexes \(\rho^f:\mathrm{Def}(\lambda^f)[-1]\rightarrow\mathrm{Def}(\mu^f)\) by \(\gamma\mapsto \gamma\circ (f_{p,q})\).
\begin{remark}
The deformation complex \(\mathrm{Def}(f)\) is isomorphic as a chain complex to the mapping cone of \(\rho^f\) and as a graded Lie algebra to \(\mathrm{Def}(\lambda^f)\ltimes\mathrm{Def}(\mu^f)\).
\end{remark}
\bibliographystyle{amsplain}
\bibliography{qocha}

\end{document}